\documentclass{amsart}
\usepackage{amsmath,dsfont,amsfonts}
\usepackage{amssymb}
\usepackage{amsthm}
\usepackage{verbatim}

\newtheorem{teo}{Theorem}[section]
\newtheorem{prop}[teo]{Proposition}
\newtheorem{lm}[teo]{Lemma}
\newtheorem{coro}[teo]{Corollary}
\newtheorem{rem}[teo]{Remark}

\newcommand{\RR}{{\mathbb{R}}}

\newcommand{\der}{\partial}

\newcommand{\om}{\Omega}

\def\qed{\hfill$\square$\vspace{0.5cm}}    

\numberwithin{equation}{section}

\begin{document}

\title[Differentiability of the Dirichlet to Neumann map]{Differentiability of the Dirichlet to Neumann map under movements of polygonal inclusions with an application to shape optimization}

\author[E.~Beretta et al.]{Elena~Beretta}
\address{ Dipartimento di Matematica ``Brioschi'',
Politecnico di Milano, Italy}
\email{elena.beretta@polimi.it}
\author[]{Elisa~Francini}
\address{Dipartimento di Matematica e Informatica ``U. Dini'',
Universit\`{a} di Firenze, Italy}
\email{elisa.francini@unifi.it}
 \author[]{Sergio~Vessella}
 \address{Dipartimento di Matematica e Informatica ``U. Dini'',
Universit\`{a} di Firenze, Italy}
\email{sergio.vessella@unifi.it}

\date{\today}


\begin{abstract}
In this paper we derive rigorously the derivative of the Dirichlet to Neumann map and of the Neumann to Dirichlet map of the conductivity equation with respect to movements of vertices of triangular conductivity inclusions. We apply this result to formulate an optimization problem based on a shape derivative approach. 
\end{abstract}

\maketitle

\section{Introduction}
This paper contains a rigorous proof of the formula for the derivative of the normal derivative of the solution to the following boundary value problem
\begin{equation}\label{conductivity}
    \left\{\begin{array}{rcl}
             \textrm{ div }((1+(k-1)\chi_{T})\nabla u) & = & 0\mbox{ in }\om, \\
             u & = & f \mbox{ on }\der\om,
           \end{array}
    \right.
\end{equation}
with respect to affine movements of a triangular inclusion $T \subset \om$.

More precisely, let us fix two functions $f, g\in H^{1/2}(\der\om)$. Let  $P:=(P_1,P_2,P_3)\in \om\times\om\times\om$ represent the vertices of $T$. Let $F:\om\times\om\times\om\rightarrow \mathbb{R}$ denote the following function
$$
F(P)=\int_{\der\om}\frac{\der u}{\der \nu}g\,d\sigma.
$$
We prove that $F$ is differentiable and we derive rigorously an explicit formula of the differential. 
Moreover, the quantitative proof of our main result  allows us to infer differentiability of the Dirichlet to Neumann map with respect to motions of the vertices of the triangle $T$ (see Corollary \ref{differenziabilita}). An analogous result is established for the Neumann to Dirichlet map.

We want to highlight that the derived formula can be cast in the framework of shape derivatives, a tool, which has been widely used in optimization problems, (see \cite{ADK1}, \cite{ADK2}, \cite{ABKFL}, \cite{AKLZ}, \cite{ADKSTVZ}, \cite{C}, \cite{HR}). Surprisingly, to our knowledge, it does not exist in the literature a rigorous proof  in the case of polygonal inclusions. 
In fact, the formulas for shape derivatives obtained in, for example  \cite{ADK1},  \cite{HR}, using a variational approach, require at least $C^{1,1}$ regularity of the boundary of the inclusion. Hence, we think it is important to present a rigorous proof of the formula for the derivative of $F$ in the case of polygonal inclusions. 
We present the result in the case of a triangular inclusion since this allows us to face the main difficulties of the problem without burdening the presentation with excessive notation. (see Remark  \ref{2.2}). 

Our interest in the derivative of $F$ is motivated by the following inverse problem:   
let $\{T_j\}_{j=1}^N$ be an unknown triangular mesh of $\Omega$ and let 
\[
\sigma=\sum_{j=1}^N\sigma_j\chi_{T_j},
\]
we would like to recover the mesh from  Neumann boundary measurements of solutions of 
\begin{equation*}
    \left\{\begin{array}{rcl}
             \textrm{ div }(\sigma\nabla u) & = & 0\mbox{ in }\om, \\
             u & = & f \mbox{ on }\der\om.
           \end{array}
    \right.
\end{equation*}
A similar inverse problem was considered  in \cite{BFdeHV}. There, the authors considered the Helmholtz equation at fixed low frequency in a bounded three-dimensional domain. Assuming  the wavespeed piecewise constant on an unknown regular tetrahedral partition  of  $\Omega$  they established a quantitative   of the Hausdorff distance between partitions in terms of  the norm of the difference of the corresponding Dirichlet to Neumann maps. One of the main ingredients in it is differentiability of the Dirichlet to Neumann map with respect to motions of the mesh.
In order to extend the stability result derived in \cite{BFdeHV} to the case of piecewise constant conductivities in dimension two a crucial step concerns the differentiability of the Dirichlet to Neumann map with respect to motions of vertices of the mesh $\{T^j\}_{j=1}^N$, $N\geq 2$. It is worthwhile to notice that this case compared to the Helmholtz equation is much more difficult since solutions are less regular. In particular,  the gradient of solutions may blow up at vertices of the mesh. 

The main tools to establish our result are energy estimates, fine elliptic regularity results for solutions to transmission problems for elliptic systems and equations obtained in \cite{LN}, \cite{LV} and the study of the exact asymptotic behaviour of gradients of solutions in a neighborhood of vertices (see \cite{BFI}). 
As we mentioned above the result is important in the context of  shape optimization (see \cite{SZ}, \cite{DZ}), since it generalizes the computation of the shape derivative to a class of non smooth inclusions. Furthermore,  in Section \ref{section4} we  describe how the result can be used to formulate a reconstruction algorithm for the inverse problem of recovering a polygonal inclusion from boundary measurements based on minimization of a suitable boundary functional. A similar approach has been used in \cite{ADK1}, \cite{ADKSTVZ} and \cite{HR} in the case of smooth inclusions.

The plan of the paper is the following: in Section \ref{section1} we present our main assumptions and the main results. In Section \ref{section2} we state some preliminary results concerning the local behaviour of solutions to the conductivity equation at the interface of the triangle. Based on these results we prove some energy estimates and some pointwise estimates of the gradient of the difference between the solutions of the conductivity equation corresponding to a fixed triangle $T$ and a perturbation of such a solution under a small  motion of the vertices of the triangle $T$, $T^t$. In Section \ref{section3} we give the proof of our main result that we split in several steps: first we derive the formula of the derivative under suitable movements that move only one vertex of the triangle.  Then, we extend the validity of the formula to arbitrary movements of one vertex. Finally, we show the general result by superposition of three displacements. 
As mentioned above, in the final Section we present an application to shape optimization.
\section{Notations, assumptions and main result}\label{section1}
\begin{center}
\textit{Main Assumptions}
\end{center}
Let $\om$ be a bounded domain in $\RR^2$ and $\partial\Omega\in C^{0,1}$
\begin{equation*}
\textrm{diam\,}\om\leq L ,
\end{equation*}
Let $d_0>0$ be such that
\begin{equation*}
\om_{d_0}=\{(x,y)\in\om: \textrm{dist}((x,y),\der\om)\geq d_0\}
\end{equation*}
is a nonempty subset of $\om$. Let $T^P$ denote an open triangle of vertices $P=(P_1,P_2,P_3)$ where $P$ belongs to the following subset of $\RR^2\times\RR^2\times\RR^2$
\begin{equation*}
\mathcal{V}=\left\{P\in \om_{d_0}\times\om_{d_0}\times\om_{d_0}:|P_i-P_j|\geq \alpha_1, \theta_i\geq \alpha_0,\, i,j=1,2,3\right\}
\end{equation*}
where $\theta_i$ denotes the width of the angle at the vertex $P_i$ and $\alpha_0,\alpha_1$ are positive given constants.
Let, for $0<k$, $k\neq 1$
\begin{equation*}
\sigma_P=1+(k-1)\chi_{T^P}\quad \textrm{in }\Omega.
\end{equation*}
Given $f\in H^{1/2}\left(\der\om\right)$ we denote by $u_P\in H^1(\om)$ the unique weak solution to 
\begin{equation}\label{probP}
    \left\{\begin{array}{rcl}
             \textrm{ div }(\sigma_P\nabla u_P) & = & 0\mbox{ in }\om, \\
             u_P & = & f \mbox{ on }\der\om.
           \end{array}
    \right.
\end{equation}
Define the Dirichlet to Neumann map
\begin{equation*}
    \Lambda_{\sigma_P}: H^{1/2}\left(\der\om\right)\to H^{-1/2}\left(\der\om\right)
\end{equation*}
as follows: given $f \in H^{1/2}\left(\der\om\right)$,  
\begin{equation*}
    \Lambda_{\sigma_P}(f):={\frac{\der u_P}{\der \nu}}\in  H^{-1/2}\left(\der\om\right),
\end{equation*}
where $\nu$ is the unit outer normal to $\der\om$.

For $f,g \in H^{1/2}\left(\der\om\right)$ let us define $F:\mathcal{V}\rightarrow \RR$ as follows
\[
F(P)=< \Lambda_{\sigma_P}(f),g>\quad\forall P\in \mathcal{V}
\]
where $<\cdot,\cdot>$ is the duality pairing between $ H^{1/2}\left(\der\om\right)$ and its dual   $ H^{-1/2}\left(\der\om\right)$. Consider now $P^0=(P_1^0,P_2^0,P_3^0) \in \mathcal{V}$ and the corresponding triangle  denoted by $T^0:=T^{P^0}$. Let $\vec V=(\vec V_1,\vec V_2,\vec V_3)\in \RR^2\times\RR^2\times\RR^2$ be an arbitrary  vector and let $\Phi_0^{\vec V}: \RR^2\rightarrow\RR^2$ be the affine map such that
\begin{equation*}
\Phi_0^{\vec V}(P_i^0)=\vec V_i,\,\,\,\, i=1,2,3.
\end{equation*}
Consider now $P^t=P^0+t\vec V=(P_1^t,P_2^t,P_3^t)$ for $t$ sufficiently small (for example  $|t|<d_0/(2|\vec V|)$) and the triangle of vertices $P^t$ that we denote by $T^t:=T^{P^t}$ and let 
\[G(t):=F(P^t).\] 
Then, let $u_0:=u_{P^0}$ the solution of (\ref{probP}) corresponding to 
\[\sigma_0:=\sigma_{P^0}=1+(k-1)\chi_{T^0},\] denote by $v_0$ the solution of 
\begin{equation*}
    \left\{\begin{array}{rcl}
             \textrm{ div }(\sigma_0\nabla v_0) & = & 0\mbox{ in }\om, \\
              v_0& = & g \mbox{ on }\der\om.
           \end{array}
    \right.
\end{equation*}
Finally, denote by $u_t:=u_{P^t}$ the solution of (\ref{probP}) corresponding to \[\sigma_t:=\sigma_{P^t}=1+(k-1)\chi_{T^t}\] and by $v_t$ the solution to 
\begin{equation}\label{probPt}
    \left\{\begin{array}{rcl}
             \textrm{ div }(\sigma_t\nabla v_t) & = & 0\mbox{ in }\om, \\
              v_t& = & g \mbox{ on }\der\om.
           \end{array}
    \right.
\end{equation}
Let us define 
\[
u_0^e=u_0|_{\om\backslash T^0},\,\,\,\, v_0^e=v_0|_{\om\backslash T^0}.
\]
 Fix an orthonormal system $(\tau_0, n_0)$ in such  a way that $n_0$ represents the outward unit normal to $\der T^0\backslash P^0$, the tangent unit vector $\tau_0$ is oriented clockwise and denote by $M_0 $ a $2\times 2$ matrix valued function defined on $\partial T^0$ with eigenvalues $1$ and $\frac{1}{k}$ and corresponding eigenvectors $\tau_0$ and $n_0$.

Our main result is the following 
\begin{teo}\label{MainResult} There exist two positive constants $C$ and $\alpha$ depending only on
 $L$, $ d_0$, $\alpha_0$, $\alpha_1$, $k$, such that
\begin{equation}\label{shapederivative}
\begin{aligned}
\left|G(t)-G(0)-\vphantom{\int_{\der T^0}}\right.&\left.t(k-1)\int_{\der T^0}(M_0\nabla u_0^e\cdot\nabla v_0^e)(\Phi_0^{\vec V}\cdot n_0)\,d\sigma\right|\\
&\leq Ct^{1+\alpha}\|f\|_{H^{1/2}(\der\om)}\|g\|_{H^{1/2}(\der\om)},
\end{aligned}
\end{equation}
for $|t|<d_0/(2|\vec V|)$.
\end{teo}
An obvious consequence of Theorem \ref{MainResult} is:
\begin{coro} \label{remnuovo1}For given $f$ and $g$ in $H^{1/2}(\der\om)$, $G$ is differentiable and
\begin{equation}\label{shapederivativebis}
G^\prime(0)=(k-1)\int_{\der T^0}(M_0\nabla u_0^e\cdot\nabla v_0^e)(\Phi_0^{\vec V}\cdot n_0)\,d\sigma.
\end{equation}
\end{coro}

\vskip 2truemm
\begin{rem}\label{2.2}
We want to point out that Theorem \ref{MainResult} extends to the case of a polygonal inclusion. 
Let $\mathcal{T}^{0}$ be a polygon that, for simplicity, we assume it is convex,
 of $N$ vertices, $P^0=(P_1^0, \dots, P_N^0)$ ordered clockwise.
Assume that
\[
\mathcal{T}^{0}\subset \Omega_{d_0},
\]
there exists $\alpha_0\in (0,\pi/2)$ such that, denoting by $\theta_i$ the interior angle  at vertex $P_i$,  
\[
\alpha_0\leq\theta_i\leq\pi-\alpha_0, \,\,\,\forall i=1,\dots, N.
\]

Assume also there exists $\alpha_1>0$ such that
\[
|P_{i+1}-P_i|>\alpha_1,\,\,\forall i=1,\dots,N,
\]
where we set $P_{N+1}:=P_1$.

Denote by $P^t=P^0+t\vec V$ where $\vec V=(\vec V_1,\dots,\vec V_N)$, $\vec V_i\in\mathbb{R}^2$ for all $i=1,\dots,N$ and $t\in\mathbb{R}$ is sufficiently small. Let  $\mathcal{T}^t$ be the polyon of vertices $P^t$ and define
\[
G(t)=< \Lambda_{\sigma_t}(f),g>
\]
where $f,g \in H^{1/2}\left(\der\om\right)$ and $\Lambda_{\sigma_t}$ denotes the Dirichlet to Neumann map of the operator 
\[{ div }(\sigma_t\nabla \cdot)\]
 with  
 \[\sigma_t=1+(k-1)\chi_{\mathcal{T}^t}.\]

We have that $G$ is differentiable at $t=0$  and $G^\prime(0)$ can be expressed by the following formula
\[
G^\prime(0)=(k-1)\int_{\der \mathcal{T}^0}(M_0\nabla u_0^e\cdot\nabla v_0^e)(\Phi^{\vec V}\cdot n_0)d\sigma,
\]
where $\Phi^{\vec V}:\der\mathcal{T}^0\rightarrow \mathbb{R}^2$ and
\[
\Phi^{\vec V}(Q)=\vec V_i+\frac{(Q-P_i^0)\cdot (P^0_{i+1}-P^0_i)}{|P^0_{i+1}-P^0_i|^2}(\vec V_{i+1}-\vec V_i),\mbox{ for }
Q\mbox{ on the side } P^0_iP^0_{i+1}.
\]

\end{rem}

\section{Preliminary results}\label{section2}
In this section we collect some preliminary results some of them concerning the regularity of solutions to Problem (\ref{probP}) and that are crucial to prove our main theorem. Let $u_P$ be the solution to Problem (\ref{probP}) with $P=(P_1,P_2,P_3)\in\mathcal{V}$. Let 
\[u_P^e={u_P}_{|_{\om\backslash T^P}},\,\,\,\, u_P^i={u_P}_{|_{\overline{T}^P}}.\]
Finally, let $B(P_j,\delta)$ denote the ball centered at $P_j$ and radius $\delta>0$.
Then, the following estimates holds true
\begin{prop}\label{LiNirenberg}
There exist positive constants $C$ and $\gamma\in (0,1/4)$ and $\delta_0$ depending only on $\alpha_0$, $\alpha_1$, $k$ and $d_0$, such that, for any $\delta \in (0,\delta_0)$ we have
\begin{equation}\label{gradest1}
\|\nabla u_P^i\|_{C^{\gamma}(T^P\backslash \cup_{j=1}^3B(P_j,\delta))}\leq \frac{C}{\delta^{\gamma+1}}\|u\|_{L^2(\Omega)}
\end{equation}
\begin{equation}\label{gradest2}
\|\nabla u_P^e\|_{C^{\gamma}(\om_{d_0}\backslash(T^P\cup( \cup_{j=1}^3B(P_j,\delta)))}\leq \frac{C}{\delta^{\gamma+1}}\|u\|_{L^2(\Omega)}
\end{equation}
\end{prop}
The proposition is consequence of a more general regularity result for elliptic systems due to Li and Nirenberg and to Li and Vogelius for elliptic equations  (cf. \cite{LN} and \cite{LV}). \\
Moreover, by a result of Bellout, Friedman and Isakov (cf. \cite{BFI}) it is possible to describe the exact behaviour of $u_P$ in a neighborhood of the vertices of $T^P$. In particular,  the following estimates hold true
\begin{prop}\label{BellFriedIs}
There exist a constant $\omega>\frac{1}{2}$ and two positive constants $\delta_0$ and  $C$ depending only on
$\alpha_0$, $\alpha_1$, $k$ and $d_0$  such that, for $j=1,2,3$
\begin{equation}\label{gradest3}
|\nabla u_P^i(x,y)|\leq C\|u\|_{L^2(\Omega)}\textrm{ dist }((x,y),P_j)^{\omega-1},
\end{equation}
\begin{equation}\label{gradest4}
|\nabla u_P^e(x,y)|\leq C\|u\|_{L^2(\Omega)}\textrm{ dist }((x,y),P_j)^{\omega-1},
\end{equation}
for all $(x,y)\in B(P_j,\delta_0)\backslash\{P_j\}$.
\end{prop}

\begin{rem}
In all the results that we state and prove from now on, the estimates depend linearly on the norms of the boundary data $f$ and $g$. For this reason, for sake of simplicity, we normalize functions $f$ and $g$ by taking
\[\|f\|_{H^{1/2}(\der\om)}=\|g\|_{H^{1/2}(\der\om)}=1\]
in the proofs, while we explicitly write the dependence on the norms in the statements.
\end{rem}

We can prove the following energy estimates
\begin{prop}\label{energy}
Let $P^0\in \mathcal{V}$ and $u_0$ be the solution of (\ref{probP}) corresponding to the triangle $T^0$ of vertices $P^0=(P^0_1,P^0_2,P^0_3)$. Let $u_t$ be the solution of (\ref{probP}) corresponding to the triangle $T^t$ of vertices $P^t=(P^t_1,P^t_2,P^t_3)$ with $P_i^t=P_i^0+t\vec V_i$, $i=1,2,3$ and $\vec V=(\vec V_1.\vec V_2, \vec V_3)$ a vector in $\RR^2\times\RR^2\times\RR^2$. Then, there exist  positive constants $C$ and $0<\theta<\frac{1}{2}$  independent of $t$, such that 
\begin{equation}\label{energyest}
\|u_t-u_0\|_{H^1(\om)}\leq C\|f\|_{H^{1/2}(\der\om)}t^{\theta},
\end{equation}
for $|t|<d_0/(2|\vec V|)$.
\end{prop}
\begin{proof}
Consider $w_t=u_t-u_0$. Then, $w_t$ solves
\begin{equation*}
    \left\{\begin{array}{rcl}
             \textrm{ div}(\sigma_t\nabla w_t) & = & \textrm{ div}((\sigma_0-\sigma_t)\nabla u_0) \mbox{ in }\om, \\
              w_t& = & 0 \mbox{ on }\der\om.
           \end{array}
    \right.
\end{equation*}
Multiplying the above equation by $w_t$, integrating over $\om$ and integrating by parts we get
\begin{equation}\label{energyest1}
\int_{\om}\sigma_t|\nabla w_t|^2=\int_{T^t \triangle T^0}(\sigma_0-\sigma_t)\nabla u_0\cdot\nabla w_t
\end{equation}
and by 
\[
\sigma_t\geq \min\{1,k\}
\]
we get
\begin{eqnarray*}
\int_{\om}|\nabla w_t|^2&\leq& \frac{|k-1|}{\min\{1,k\}}\int_{T^t \triangle T^0}|\nabla u_0||\nabla w_t|\\
&\leq& \frac{|k-1|}{\min\{1,k\}}\left(\int_{T^t \triangle T^0}|\nabla u_0|^2\right)^{1/2}\left(\int_{T^t \triangle T^0}|\nabla w_t|^2\right)^{1/2}\\
&\leq& \frac{|k-1|}{\min\{1,k\}}\left(\int_{T^t \triangle T^0}|\nabla u_0|^2\right)^{1/2}\left(\int_{\om}|\nabla w_t|^2\right)^{1/2}.
\end{eqnarray*}
Hence,
\begin{equation*}
\int_{\om}|\nabla w_t|^2\leq \left(\frac{|k-1|}{\min\{1,k\}}\right)^2\int_{T^t \triangle T^0}|\nabla u_0|^2
\end{equation*}
Let   $\mathcal{S}_{\delta_t}=\cup_{j=1}^3B(P_j^0,\delta_t)$ with $\delta_t>0$ to be chosen later. Using (\ref{gradest1}) and (\ref{gradest2}) we obtain the following bounds
\begin{eqnarray*}
\int_{T^t \triangle T^0}|\nabla u_0|^2&=&\int_{(T^t \triangle T^0)\cap \mathcal{S}_{\delta_t}}|\nabla u_0|^2+\int_{(T^t \triangle T^0)\backslash \mathcal{S}_{\delta_t}}|\nabla u_0|^2\\
&\leq & \int_{(T^t \triangle T^0)\cap \mathcal{S}_{\delta_t}}|\nabla u_0|^2+\sup_{(T^t \triangle T^0)\backslash\mathcal{S}_{\delta_t}}|\nabla u_0|^2|T^t \triangle T^0|\\
&\leq &\int_{(T^t \triangle T^0)\cap \mathcal{S}_{\delta_t}}|\nabla u_0|^2+\frac{Ct}{\delta_t^{2(\gamma+1)}}\\
&\leq &\int_{(T^t \triangle T^0)\cap \mathcal{S}_{\delta_t}}|\nabla u_0|^2+\frac{Ct}{\delta_t^{2(\gamma+1)}}.
\end{eqnarray*}
Now, using (\ref{gradest3}), for any $j=1,2,3$, we get
\begin{eqnarray*}
\int_{T^t \triangle T^0\cap B(P_j^0,\delta_t)}|\nabla u_0|^2&\leq& C\int_{B(P_j^0,\delta_t)}(\textrm{dist}((x,y),P_j))^{2(\omega-1)}\\
&\leq & C\int_0^{\theta_j}\int_0^{\delta_t}\rho^{2(\omega-1)}\rho\, d\rho d \theta\leq C\int_0^{\delta_t}\rho^{2\omega-1}\, d\rho\leq C\delta_t^{2\omega}.
\end{eqnarray*}
In conclusion, we have
\[
\int_{T^t \triangle T^0}|\nabla u_0|^2\leq C\left(\delta_t^{2\omega}+\frac{t}{\delta_t^{2(\gamma+1)}}\right).
\]
Picking up $\delta_t=t^{\alpha}$ with $\alpha=\frac{1}{2(\omega+\gamma+1)}$ we derive
\[
\int_{T^t \triangle T^0}|\nabla u_0|^2\leq Ct^{\frac{\omega}{\omega+\gamma+1}}.
\]
Finally, by last inequality and by (\ref{energyest1}) we obtain
\[
\left(\int_{\om}|\nabla w_t|^2\right)^{1/2}\leq Ct^{\frac{\omega}{2(\omega+\gamma+1)}}=Ct^{\theta}
\]
with $0<\theta<\frac{1}{2}$ which ends the proof.

\end{proof}
We have also the following
\begin{prop}\label{pointwisegrad}
Let $\mathcal{L}=\bigcup_{j=1}^3\left(B(P_j^0,\delta_t)\cup B(P_j^t,\delta_t)\right)$ 
with $\delta_t=t^{\beta_1}$ and  $\beta_1=\frac{\theta\gamma}{2(\gamma+1)}$ and let $u_t$ and $u_0$ be defined as in Proposition \ref{energy}. Then, there exists a positive constant $C$, independent of $t$, such  that
\begin{equation}\label{pointwisest1}
\|\nabla u^e_t-\nabla u^e_0\|_{L^{\infty}(\der T^t\backslash(T^0\cup\mathcal{L})}+\|\nabla u^i_t-\nabla u^i_0\|_{L^{\infty}(\der T^t\cap T^0\backslash\mathcal{L})}\leq C\|f\|_{H^{1/2}(\der\om)} t^{\beta_1}
\end{equation}
and
\begin{equation}\label{pointwisest2}
\|\nabla u^e_t-\nabla u^e_0\|_{L^{\infty}(\der T^0\backslash(T^t\cup\mathcal{L})}+\|\nabla u^i_t-\nabla u^i_0\|_{L^{\infty}(\der T^0\cap T^t\backslash\mathcal{L})}\leq C \|f\|_{H^{1/2}(\der\om)}t^{\beta_1},
\end{equation}
\end{prop}
for $|t|<d_0/(2|\vec V|)$.
\begin{proof}
It is sufficient to show the first inequality (\ref{pointwisest1}) since (\ref{pointwisest2}) follows similarly. Let $t<d<\frac{d_0}{2}$ and denote by  
\[
\om_d^t=\{(x,y)\in\om\backslash (T^t\cup T^0): \textrm{dist}((x,y),\der(\om\backslash(T^t\cup T^0)))\geq d\}.
\]
Observe that since $\nabla (u^e_t- u^e_0)$ is harmonic in $\om\backslash(T^t\cup T^0)$, by the mean value theorem we get
\[
\|\nabla u^e_t-\nabla u^e_0\|_{L^{\infty}(\om_d^t)}\leq \frac{C}{d}\|\nabla u_t-\nabla u_0\|_{L^{2}(\om)}
\]
Then, by (\ref{energyest}) we obtain
\begin{equation}\label{extest}
\|\nabla u^e_t-\nabla u^e_0\|_{L^{\infty}(\om_d^t)}\leq \frac{C}{d}t^{\theta}.
\end{equation}
Let now $(x,y)\in \der T^t\backslash(T^0\cup\mathcal{L})$ and let $(x_d,y_d)$ be the closest point to $(x,y)$ in $\om_d^t$. Then, by Proposition \ref{LiNirenberg}, we have
\[
|\nabla u_t^e(x,y)-\nabla u_t^e(x_d,y_d)|\leq C\frac{d^{\gamma}}{\delta_t^{\gamma+1}}
\]
and
\[
|\nabla u_0^e(x,y)-\nabla u_0^e(x_d,y_d)|\leq C\frac{d^{\gamma}}{\delta_t^{\gamma+1}}.
\]
Combining last two inequalities with (\ref{extest}) we get
\[
|\nabla u_t^e(x,y)-\nabla u_0^e(x,y)|\leq C\left(\frac{d^{\gamma}}{\delta_t^{\gamma+1}}+\frac{t^{\theta}}{d}\right).
\]
By choosing $d=t^{\frac{\theta+\beta_1(\gamma+1)}{\gamma+1}}$ and  $\delta_t=t^\frac{\theta\gamma}{2(\gamma+1)}$ we obtain, for $(x,y)\in \der T^t\backslash (T^0\cup\mathcal{L})$, 
\[
|\nabla u_t^e(x,y)-\nabla u_0^e(x,y)|\leq Ct^{\frac{\theta\gamma}{2(\gamma+1)}}
\]
so that 
\[
\|\nabla u_t^e-\nabla u_0^e\|_{L^{\infty}( \der T^t\backslash (T^0\cup\mathcal{L}))}\leq C t^{\frac{\theta\gamma}{2(\gamma+1)}}. 
\]

Let now $(x,y)\in (\der T^t\cap T^0)\backslash \mathcal{L}$ and define the set
\[
(T^t\cap T^0)_d=\{(x,y)\in T^t\cap T^0: \textrm{ dist}((x,y),\der T^t\cup \der T^0)\geq d\}
\]
and let $(x_d,y_d)\in (T^t\cap T^0)_d $  be the closest point to $(x,y)$. Since, $\nabla (u^i_t- u^i_0)$ is harmonic in $T^t\cap T^0$ we get, applying the mean value theorem and using (\ref{energyest})
\[
\|\nabla u^i_t-\nabla u^i_0\|_{L^{\infty}((T^t\cap T^0)_d)}\leq \frac{C}{d}t^{\theta}
\]
and by (\ref{gradest1}) we get
\[
|\nabla u^i_t(x,y)-\nabla u^i_0(x,y)|\leq \left(\frac{d^{\gamma}}{\delta_t^{\gamma+1}}+\frac{t^{\theta}}{d}\right)
\]
and choosing $d=t^{\frac{\theta+\beta_1(\gamma+1)}{\gamma+1}}$ and  $\delta_t=t^\frac{\theta\gamma}{2(\gamma+1)}$ we finally derive
\[
\|\nabla u^i_t-\nabla u^i_0\|_ {L^{\infty}((\der T^t\cap T^0)\backslash\mathcal{L})}\leq C t^{\frac{\theta\gamma}{2(\gamma+1)}}
\]
which concludes the proof.
\end{proof}

\section{Proof of the main result}\label{section3}
In order to prove our main result we will first establish the validity of (\ref{shapederivative}) for particular choices of the vector $\vec V$.
Indeed, we first consider a special direction $\vec V=(\vec V_1,\vec 0,\vec 0)$ where $\vec V_1$ is $\frac{P_2^0-P_1^0}{|P_2^0-P_1^0|}$ or $\frac{P_3^0-P_1^0}{|P_3^0-P_1^0|}$. The reason for this choice is that $T^0\triangle T^t$, that is the support of $\sigma_t-\sigma_0$, is easier to describe. Then we consider $\vec V=(\vec v_1,\vec 0,\vec 0)$ where $\vec v_1$ is an arbitrary vector that we decompose as a linear combination of the vectors $\frac{P_2^0-P_1^0}{|P_2^0-P_1^0|}$ and $\frac{P_3^0-P_1^0}{|P_3^0-P_1^0|}$. In order to perform this step we need to show that the functional 
\[
D(t):=\int_{\partial T^t}M_t\nabla u_t^e\cdot\nabla v_t^e(\Phi_t^{\vec V}\cdot \vec n_t)d\sigma
\]
is continuous at $t=0$. The last step is just based on linearity of the limit process.

\begin{lm}\label{sidederivative}
Let $\vec V=(\vec V_1,\vec 0,\vec 0)$ where $\vec V_1=\frac{P_2^0-P_1^0}{|P_2^0-P_1^0|}, \left(\frac{P_3^0-P_1^0}{|P_3^0-P_1^0|}\right)$.
Then
\begin{equation}\label{sidederivative0}
\begin{aligned}
\left|G(t)-G(0)-\vphantom{\int_{\der T^0}}\right.&\left.t(k-1)\int_{\der T^0}(M_0\nabla u_0^e\cdot\nabla v_0^e)(\Phi_0^{\vec V}\cdot n_0)\,d\sigma\right|\\
&\leq Ct^{1+\alpha}\|f\|_{H^{1/2}(\der\om)}\|g\|_{H^{1/2}(\der\om)}
\end{aligned}
\end{equation}
for $|t|<d_0/2$.
\end{lm}
\begin{proof}
Without loss of generality let $\vec V=\left(\frac{P_2^0-P_1^0}{|P_2^0-P_1^0|},\vec 0,\vec 0\right)$ and consider $P^t=P^0+t\vec V$. By Alessandrini's identity
\begin{equation}\label{sidederivative1}
\frac{G(t)-G(0)}{t}=\frac{1}{t}<(\Lambda_t-\Lambda_0)f,g>=\frac{1}{t}\int_{\om}(\sigma_t-\sigma_0)\nabla u_t\cdot\nabla v_0
\end{equation}
where $\Lambda_t:=\Lambda_{\sigma_t}$ and $\Lambda_0:=\Lambda_{\sigma_0}$ and $u_t$ is solution to Problem (\ref{probP}) corresponding to $\sigma_t$ and $v_0$ the solution to Problem (\ref{probPt}) corresponding to $\sigma_0$. We choose a coordinate frame in such a way that $P_1^0=(0,0)$, $P_2^0=(-x_2,0)$, $x_2>0$,  $P_3^0=(x_3,y_3)$ with $y_3<0$. Then, $\vec V=(-\vec e_1, \vec 0, \vec 0)$ and $P^t=((-t,0),(-x_2,0),(x_3,y_3))$. Assume $t>0$ (the case $t<0$ can be treated similarly) and consider the set
\[
\mathcal{C}_{\delta_t}=T^0\backslash \left(T^t\cup\{(x,y): -\delta_t<y<y_3+\delta_t\}\right),
\]
where $\delta_t$ will be chosen later. For every $(x,y)\in\mathcal{C}_{\delta_t}$ let $(x(y),y)=(\frac{x_3}{y_3}y,y)$ be the corresponding point on the side $P_2^0P_1^0$ of $T^0$. By Proposition \ref{LiNirenberg} applied to $u_t^e$ and $v_0^i$ we get
\begin{equation}\label{approxgrad}
\begin{array}{rcl}
\nabla u_t^e(x,y)&=&\nabla u_t^e(x(y),y)+R_1(x,y)\\
\nabla v_0^e(x,y)&=&\nabla v_0^e(x(y),y)+R_2(x,y)
\end{array}
\end{equation}
where 
\[
|R_i(x,y)|\leq C\frac{|x-x(y)|^{\gamma}}{\delta_t^{\gamma+1}},\,\,\,i=1,2.
\]
Since 
\[|x-x(y)|\leq |x_1(y)-x(y)|\leq t\frac{|y-y_3|}{|y_3|},
\] 
where $x_1(y)=\frac{x_3+t}{y_3}y-t$, we get
\begin{equation}\label{reminderbound}
|R_i(x,y)|\leq C\frac{t^{\gamma}}{\delta_t^{\gamma+1}}\frac{|y-y_3|}{|y_3|},\,\,\,i=1,2.
\end{equation}
Using (\ref{sidederivative1}) we can write
\begin{equation}\label{sidederivative2}
\frac{G(t)-G(0)}{t}=\frac{1-k}{t}\int_{\mathcal{C}_{\delta_t}}\nabla u^e_t\cdot\nabla v^i_0+\frac{1-k}{t}\int_{T^0\backslash(T^t\cup\mathcal{C}_{\delta_t})}\nabla u^e_t\cdot\nabla v^i_0=I_1+I_2
\end{equation}
Let us first estimate $I_1$. We use (\ref{gradest3}) and (\ref{gradest4}) and we split the integral over $\mathcal{C}_{\delta_t}=\mathcal{C}^1_{\delta_t}\cup\mathcal{C}^2_{\delta_t}$ in two parts
\begin{eqnarray*}
\left |\frac{1}{t}\int_{\mathcal{C}^1_{\delta_t}}\nabla u^e_t\cdot\nabla v^i_0\right|&=&\left |\frac{1}{t}\int_{-\delta_t}^0\left(\int_{x_1(y)}^{x(y)}\nabla u^e_t\cdot\nabla v^i_0\,dx\right)dy\right|\\
&\leq& \frac{C}{t}\int_{-\delta_t}^0\left(\int_{-t}^{0}\frac{dx}{((x+t)^2+y^2)^{\frac{1-\omega}{2}}(x^2+y^2)^{\frac{1-\omega}{2}}}\right)dy.
\end{eqnarray*}
By the change of variables $x=tX$ and $y=tY$ we get
\begin{eqnarray*}
\left |\frac{1}{t}\int_{\mathcal{C}^1_{\delta_t}}\nabla u^e_t\cdot\nabla v^i_0\right|&\leq& \!\!\! C\frac{t^{2\omega}}{t}\int_{-\delta_t/t}^0\left(\int_{-1}^{0}\frac{dX}{((X+1)^2+Y^2)^{\frac{1-\omega}{2}}(X^2+Y^2)^{\frac{1-\omega}{2}}}\right)dY\\
&\leq &\!\!\!C\delta_t^{2\omega-1}\int ^0_{-\frac{\delta_t}{t}}\frac{dY}{|Y|^{2-2\omega}}\leq Ct^{2\omega-1}\left(\frac{\delta_t}{t}\right)^{2\omega-1}
\leq C\delta_t^{2\omega-1} 
\end{eqnarray*}
and proceeding similarly one can see that
\begin{eqnarray*}
\left |\frac{1}{t}\int_{\mathcal{C}^2_{\delta_t}}\nabla u^e_t\cdot\nabla v^i_0\right|&\leq &\left |\frac{C}{t}\int_{y_3}^{y_3+\delta_t}\left(\int_{x_1(y)}^{x(y)}\frac{dx}{((x-x_3)^2+(y-y_3)^2)^{1-\omega}}\right)dy\right|\\
&\leq& C \int_{y_3}^{y_3+\delta_t}\frac{dy}{((x(y)-x_3)^2+(y-y_3)^2)^{1-\omega}} \\&\leq& C \int_{y_3}^{y_3+\delta_t}\frac{dy}{(y-y_3)^{2-2\omega}}\leq C\delta_t^{2\omega-1}.
\end{eqnarray*}
Hence, alltogether we end up with
\begin{equation}\label{estI1}
|I_1|\leq C \delta_t^{2\omega-1}.
\end{equation}
Let us now consider $I_2$ and let $B_t=T^0\backslash(T^t\cup \mathcal{C}_{\delta_t})$. Inserting (\ref{approxgrad}) into $I_2$ leads to
\begin{eqnarray*}
I_2\!\!&=&\frac{1-k}{t}\left\{\int_{B_t}\nabla u^e_t(x(y),y)\cdot\nabla v^i_0(x(y),y)\,dxdy\right. \\
\!\!&+&\!\!\!\!\!\left.\int_{B_t}\!\!(R_1(x,y)\cdot\nabla v^i_0(x,y)+\nabla u^e_t(x,y)\cdot R_2(x,y)+R_1(x,y)\cdot R_2(x,y))\,dxdy\right\}\\&=:&I_3+I_4
\end{eqnarray*}
Let us evaluate $I_3$. We know that
\begin{eqnarray*}
I_3&=&\frac{(1-k)}{t}\int_{B_t}\nabla u^e_t(x(y),y)\cdot\nabla v^i_0(x(y),y)\,dxdy\\&=&\frac{(1-k)}{t}\int_{y_3+\delta_t}^{-\delta_t}\nabla u^e_t(x(y),y)\cdot\nabla v^i_0(x(y),y)(x(y)-x_1(y)) \,dy\\
&=&(1-k)\int_{y_3+\delta_t}^{-\delta_t}\nabla u^e_t(x(y),y)\cdot\nabla v^i_0(x(y),y)\left(\frac{y-y_3}{-y_3}\right) \,dy\\
&=&(1-k)\int_{y_3+\delta_t}^{-\delta_t}(\nabla u^e_0(x(y),y)+\rho(y,t))\cdot\nabla v^i_0(x(y),y)\left(\frac{y-y_3}{-y_3}\right)dy \\
\quad(\rho(y,t)&:=&\nabla u_t^e(x(y),y)-\nabla u_0^e(x(y),y))\\
&=&(1-k)\left\{\int_{y_3+\delta_t}^{-\delta_t}\nabla u^e_0(x(y),y)\cdot\nabla v^i_0(x(y),y)\left(\frac{y-y_3}{-y_3}\right)dy\right. + \\
&+&\left. \int_{y_3+\delta_t}^{-\delta_t}\rho(y,t)\cdot\nabla v^i_0(x(y),y)\left(\frac{y-y_3}{-y_3}\right)dy\right\}\\
&=& (1-k)\left\{\int_{y_3}^0\nabla u^e_0(x(y),y)\cdot\nabla v^i_0(x(y),y)\left(\frac{y-y_3}{-y_3}\right)dy\right.-\\
&-&\int_{[-\delta_t,0]\cup [y_3+\delta_t,y_3]}\nabla u^e_0(x(y),y)\cdot\nabla v^i_0(x(y),y)\left(\frac{y-y_3}{-y_3}\right)dy+\\
&+&\left. \int_{y_3+\delta_t}^{-\delta_t}\rho(y,t)\cdot\nabla v^i_0(x(y),y)\left(\frac{y-y_3}{-y_3}\right)dy\right\}.
\end{eqnarray*}
Observe now that
\begin{eqnarray*}
\left|\int_{[-\delta_t,0]\cup [y_3+\delta_t,y_3]}\nabla u^e_0(x(y),y)\cdot\nabla v^i_0(x(y),y)\left(\frac{y-y_3}{-y_3}\right)dy\right|\\
\leq \int_{[-\delta_t,0]\cup [y_3+\delta_t,y_3]}|\nabla u^e_0(x(y),y)||\nabla v^i_0(x(y),y)|dy.
\end{eqnarray*}
From Proposition \ref{BellFriedIs} and using the fact that $\omega>1/2$,  we get
\begin{eqnarray*}
\int_{[-\delta_t,0]}|\nabla u^e_0(x(y),y)||\nabla v^i_0(x(y),y)|dy&\leq &C\int_{[-\delta_t,0]}\frac{dy}{(y^2+(x(y))^2)^{1-\omega}}\\
&\leq &C\int_{[-\delta_t,0]}\frac{dy}{|y|^{2(1-\omega)}}\leq C\delta_t^{2\omega-1}.
\end{eqnarray*}
For $\int_{ [y_3+\delta_t,y_3]}|\nabla u^e_0(x(y),y)||\nabla v^i_0(x(y),y)|dy$ we can proceed similary getting 
\[
\int_{ [y_3+\delta_t,y_3]}|\nabla u^e_0(x(y),y)||\nabla v^i_0(x(y),y)|dy\leq C\delta_t^{2\omega-1}.
\]
Hence, we have
\begin{eqnarray*}
I_3&=&(1-k)\left\{\int_{y_3}^0\nabla u^e_0(x(y),y)\cdot\nabla v^i_0(x(y),y)\left(\frac{y-y_3}{-y_3}\right)dy\right. +\\
&+&\left. \int_{y_3+\delta_t}^{-\delta_t}\rho(y,t)\cdot\nabla v^i_0(x(y),y)\left(\frac{y-y_3}{-y_3}\right)dy+O(\delta_t^{2\omega-1})\right\}.
\end{eqnarray*}
Note that the following estimate holds
\begin{eqnarray}\label{v0}
\int_{y_3+\delta_t}^{-\delta_t}|\nabla v^i_0(x(y),y)|dy&\leq&\int_{y_3+\delta_t}^{y_3/2}\frac{dy}{(x(y)-x_3)^2+(y-y_3-\delta_t)^2)^{\frac{1-\omega}{2}}}+\\ \nonumber
&&+  \int_{y_3/2}^{-\delta_t}\frac{dy}{(x(y)^2+y^2)^{\frac{1-\omega}{2}}}\leq C.
\end{eqnarray}
Applying Proposition \ref{pointwisegrad} for $\delta_t=t^{\beta_1}$ with $\beta_1=\frac{\theta\gamma}{2(\gamma+1)}$ we have
\[
|\rho(y,t)|\leq C t^{\beta_1}.
\]
By last inequality and by (\ref{v0}) we derive
\[
\left|\int_{y_3+\delta_t}^{-\delta_t}\rho(y,t)\cdot\nabla v^i_0(x(y),y)\left(\frac{y-y_3}{-y_3}\right)dy\right|\leq  Ct^{\beta_1}.
\]
Eventually, we get
\begin{equation}\label{estI3}
I_3=(1-k)\int_{y_3}^0\nabla u^e_0(x(y),y)\cdot\nabla v^i_0(x(y),y)\left(\frac{y-y_3}{-y_3}\right)dy+O(t^{\beta_2})
\end{equation}
where $\beta_2=\min\{\beta_1, (2\omega-1)\beta_1\}$. 

Finally, let us evaluate $I_4$.

Using the estimates (\ref{gradest3}),  (\ref{gradest4}), (\ref{pointwisest1}), (\ref{approxgrad}) and (\ref{reminderbound}) and choosing  $\delta_t=t^{\beta_1}$ with $\beta_1=\frac{\theta\gamma}{2(\gamma+1)}$ as in Proposition \ref{pointwisegrad} we can estimate $I_4$ in the following way
\begin{eqnarray*}
|I_4|\!\!\!\!&\leq&\!\!\!\! \frac{C}{t}\int_{y_3+\delta_t}^{-\delta_t}dy\int_{x_1(y)}^{x(y)}\left\{\frac{t^{\gamma}}{\delta_t^{\gamma+1}}\left|\frac{y-y_3}{y_3}\right|^{\gamma}(|\nabla u_t^e(x(y),y)|+\right.\\
&&\left.\hphantom{aaaaaaaaaaaaaaaaaaaa}+|\nabla v_0^i(x(y),y)|)+\frac{t^{2\gamma}}{\delta_t^{2(\gamma+1)}}\left|\frac{y-y_3}{y_3}\right|^{2\gamma}\right\}dx\\
&\leq &\!\!\!\!\frac{C}{t}\int_{y_3+\delta_t}^{-\delta_t}\!\!\!\!\!|x(y)-x_1(y)|\left\{\frac{t^{\gamma}}{\delta_t^{\gamma+1}}(|\nabla u_t^e(x(y),y)|+|\nabla v_0^i(x(y),y)|)+\frac{t^{2\gamma}}{\delta_t^{2(\gamma+1)}}\right\}dy\\
&\leq &\!\!\!\!C\int_{y_3+\delta_t}^{-\delta_t}\frac{t^{\gamma}}{\delta_t^{\gamma+1}}(|\nabla u_t^e(x(y),y)|+|\nabla v_0^i(x(y),y)|)\,dy+C\left(\frac{t^{\gamma}}{\delta_t^{\gamma+1}}\right)^2.\\
\end{eqnarray*}
Finally, using the fact that, by Proposition \ref{pointwisegrad}
\[
|\nabla u_t^e(x(y),y)|\leq |\nabla u_0^e(x(y),y)|+Ct^{\beta_1}
\]
and noting that, proceeding similarly as in (\ref{v0}), we have  $\int_{y_3+\delta_t}^{-\delta_t}|\nabla u_0^e(x(y),y)|\leq C$  we get
\begin{equation}\label{estI4}
|I_4|\leq Ct^{(1-\frac{\theta}{2})\gamma}(1+t^{\beta_1}+t^{(1-\frac{\theta}{2})\gamma}).
\end{equation}
Inserting $(\ref{estI1}), $(\ref{estI3}) and (\ref{estI4}) into (\ref{sidederivative2}) we get
\[
\frac{G(t)-G(0)}{t}=(1-k)\int_{y_3}^0\nabla u^e_0(x(y),y)\cdot\nabla v^i_0(x(y),y)\left(\frac{y-y_3}{-y_3}\right)dy+r(t)
\]
where $|r(t)|\leq Ct^{\beta_3}$ where $\beta_3=\min (\beta_1,\beta_2, (1-\frac{\theta}{2})\gamma)$. Finally, observing that $\Phi_0^{\vec V}\cdot \vec n_0=\frac{y_3-y}{\sqrt{x_3^2+y_3^2}}$ on the side $P_3^0P_1^0$ and $\Phi_0^{\vec V}\cdot \vec n_0=0$ on the other sides of $T^0$ we can write
\[
\frac{G(t)-G(0)}{t}=(k-1)\int_{\partial T^0}\nabla u^e_0\cdot\nabla v^i(\Phi_0^{\vec V}\cdot \vec n_0)d\sigma+r(t).
\]
Finally, using the transmission conditions we have
\[
\nabla v_0^i=M_0\nabla v_0^e\textrm{ a.e. on } \partial T^0
\]
and we get
\begin{eqnarray*}
\frac{G(t)-G(0)}{t}&=&(k-1)\int_{\partial T^0}\nabla u^e_0\cdot M_0\nabla v^e(\Phi_0^{\vec V}\cdot \vec n_0)d\sigma+r(t)\\
&=&(k-1)\int_{\partial T^0}M_0\nabla u^e_0\cdot\nabla v^e(\Phi_0^{\vec V}\cdot \vec n_0)d\sigma+r(t)
\end{eqnarray*}
from which the claim follows.
\end{proof}
\begin{rem}
Observe that an analogous formula can be derived similarly choosing $\vec V=(\vec 0,\vec V_2,\vec 0)$ where $\vec V_2=\frac{P_1^0-P_2^0}{|P_1^0-P_2^0|}, \left(\frac{P_3^0-P_2^0}{|P_3^0-P_2^0|}\right)$ or $\vec V=(\vec 0,\vec 0,\vec V_3)$ where $\vec V_3=\frac{P_1^0-P_3^0}{|P_1^0-P_3^0|}, \left(\frac{P_2^0-P_3^0}{|P_2^0-P_3^0|}\right)$.
\end{rem}
\begin{rem}
We note that the formula (\ref{sidederivative0}) applies also to the case where $\vec V$ is not a unit vector. In fact, using the linearity of the map $\Phi^{\vec V}$ an easy computation gives 
\begin{eqnarray*}
G^\prime(0)&=&(k-1)|\vec V|\int_{\der T^0}(M_0\nabla u_0^e\cdot\nabla v_0^e)(\Phi_0^{\frac{\vec V}{|\vec V|}}\cdot n_0)d\sigma\\
&=&(k-1)\int_{\der T^0}(M_0\nabla u_0^e\cdot\nabla v_0^e)(\Phi_0^{\vec V}\cdot n_0)d\sigma
\end{eqnarray*}
\end{rem}

Let now $\vec V$ indicate an arbitrary vector of $\RR^2\times\RR^2\times\RR^2$ and let $T^0$ be a triangle of vertices $P^0\in\mathcal{V}$. Let $P^t=P^0+t\vec W$ for $t$  sufficiently small and with $\vec W=(\vec w_1,\vec 0,\vec 0)$  and let $T^t$ be the triangle of vertices  $P^t$. Consider
\[
D(t):=\int_{\partial T^t}M_t\nabla u_t^e\cdot\nabla v_t^e(\Phi_t^{\vec V}\cdot  n_t)d\sigma
\]
where $u_t$ and $v_t$ are respectively the solutions of (\ref{probP}) and (\ref{probPt}) corresponding to $\sigma_t$ and $n_t$ is the unit outer normal to $\partial T^t$. Then, we have the following
\begin{lm}\label{continuity}
There exists  constants $\beta\in (0,1)$  and $C>0$  independent on $t$ such that
\[
|D(t)-D(0)|\leq C\|f\|_{H^{1/2}(\der\om)}\|g\|_{H^{1/2}(\der\om)}t^{\beta}
\]
for $|t|<d_0/(2|\vec V|)$.
\end{lm}
\begin{proof}
Let $\mathcal{L}=\bigcup_{j=1}^3\left(B(P_j^0,\delta_t)\cup B(P_j^t,\delta_t)\right)$ with $\delta_t=t^{\beta_1}$ with $\beta_1$ defined as in Lemma \ref{sidederivative}.  Let us consider
\begin{eqnarray}\label{difference}
D(t)-D(0)\!\!\!\!&=&\!\!\!\!\int_{\partial T^t\backslash\mathcal{L}}\!\!M_t\nabla u_t^e\cdot\nabla v_t^e(\Phi_t^{\vec V}\cdot  n_t)d\sigma-\int_{\partial T^0\backslash\mathcal{L}}\!\!M_0\nabla u_0^e\cdot\nabla v_0^e(\Phi_0^{\vec V}\cdot  n_0)d\sigma\nonumber\\
&+&\!\!\!\!\int_{\partial T^t\cap\mathcal{L}}\!\!M_t\nabla u_t^e\cdot\nabla v_t^e(\Phi_t^{\vec V}\cdot  n_t)d\sigma-\int_{\partial T^0\cap\mathcal{L}}\!\!M_0\nabla u_0^e\cdot\nabla v_0^e(\Phi_0^{\vec V}\cdot  n_0)d\sigma\\
&=&\!\!\!\! J_1-J_2+J_3-J_4\nonumber
\end{eqnarray}
We assume without loss of generality that $t>0$. Let us start estimating $J_3$ and $J_4$. We accomplish this proceeding with similar calculations as in the estimation of $I_1$  in Lemma \ref{sidederivative}. In fact, we have 
\begin{equation}\label{estI3I4}
|J_3|,|J_4|\leq C \delta_t^{(2\omega-1)}\leq Ct^{(2\omega-1)\beta_1}.
\end{equation}
Let us now estimate 
\[
J_1-J_2=\int_{\partial T^t\backslash\mathcal{L}}M_t\nabla u_t^e\cdot\nabla v_t^e(\Phi_t^{\vec V}\cdot  n_t)d\sigma-\int_{\partial T^0\backslash\mathcal{L}}M_0\nabla u_0^e\cdot\nabla v_0^e(\Phi_0^{\vec V}\cdot  n_0)d\sigma.
\]
From this last difference we just consider and estimate 
\[
\int_{P_1^tP_2^0\backslash\mathcal{L}}M_t\nabla u_t^e\cdot\nabla v_t^e(\Phi_t^{\vec V}\cdot  n_t)d\sigma-\int_{P_1^0P_2^0\backslash\mathcal{L}}M_0\nabla u_0^e\cdot\nabla v_0^e(\Phi_0^{\vec V}\cdot  n_0)d\sigma
\]
since the terms on $P_1^tP_3^0$ and on $P_1^0P_3^0$ can be treated similarly. 
Let us evaluate, then,
\begin{eqnarray*}
&&\int_{P_1^tP_2^0\backslash\mathcal{L}}M_t\nabla u_t^e\cdot\nabla v_t^e(\Phi_t^{\vec V}\cdot  n_t)d\sigma-\int_{P_1^0P_2^0\backslash\mathcal{L}}M_0\nabla u_0^e\cdot\nabla v_0^e(\Phi_0^{\vec V}\cdot  n_0)d\sigma\\
&=&\int_{P_1^tP_2^0\backslash\mathcal{L}}M_t\nabla u_t^e\cdot\nabla v_t^e(\Phi_t^{\vec V}\cdot  n_t)d\sigma-\int_{P_1^0P_2^0\backslash\mathcal{L}}M_t\nabla u_t^e\cdot\nabla v_t^e(\Phi_t^{\vec V}\cdot  n_t)d\sigma\\
&+&\int_{P_1^0P_2^0\backslash\mathcal{L}}M_t\nabla u_t^e\cdot\nabla v_t^e(\Phi_t^{\vec V}\cdot  n_t)d\sigma-\int_{P_1^0P_2^0\backslash\mathcal{L}}M_0\nabla u_0^e\cdot\nabla v_0^e(\Phi_0^{\vec V}\cdot  n_0)d\sigma\\
&=:&A_1+A_2.
\end{eqnarray*}
We have
\begin{eqnarray*}
A_2&=&\int_{P_1^0P_2^0\backslash\mathcal{L}}\left\{M_t\nabla u_t^e\cdot\nabla v_t^e(\Phi_t^{\vec V}\cdot  n_t)d\sigma-M_0\nabla u_0^e\cdot\nabla v_0^e(\Phi_0^{\vec V}\cdot  n_0)\right\}d\sigma\\
&=&\int_{P_1^0P_2^0\backslash\mathcal{L}}\left\{(M_t-M_0)\nabla u_t^e\cdot\nabla v_t^e(\Phi_t^{\vec V}\cdot  n_t)+\right.\\
&&\hphantom{aaaa}+\left.M_0(\nabla u_t^e\cdot\nabla v_t^e(\Phi_t^{\vec V}\cdot  n_t)-\nabla u_0^e\cdot\nabla v_0^e(\Phi_0^{\vec V}\cdot  n_0))\right\}d\sigma\\
&=&\int_{P_1^0P_2^0\backslash\mathcal{L}}(M_t-M_0)\nabla u_t^e\cdot\nabla v_t^e(\Phi_t^{\vec V}\cdot  n_t)d\sigma\\
&&+\int_{P_1^0P_2^0\backslash\mathcal{L}}M_0(\nabla u_t^e\cdot\nabla v_t^e)(\Phi_t^{\vec V}\cdot  n_t-\Phi_0^{\vec V}\cdot  n_0)d\sigma\\
&&+ \int_{P_1^0P_2^0\backslash\mathcal{L}}M_0(\nabla u_t^e\cdot\nabla v_t^e-\nabla u_0^e\cdot\nabla v_0^e)(\Phi_0^{\vec V}\cdot  n_0)d\sigma.
\end{eqnarray*}
Using now the following estimates
\begin{eqnarray}\label{usefulbounds}
|M_t-M_0|, |\Phi_t^{\vec V}\cdot  n_t-\Phi_0^{\vec V}\cdot  n_0|&\leq&Ct,\nonumber\\
|\Phi_t^{\vec V}(Q_1)-\Phi_t^{\vec V}(Q_2)|, |\Phi_0^{\vec V}(Q_1)-\Phi_0^{\vec V}(Q_2)|&\leq&C|Q_1-Q_2|\\
|M_0|, |\Phi_0^{\vec V}\cdot  n_0|,  |\Phi_t^{\vec V}\cdot  n_t|&\leq&C\nonumber
\end{eqnarray}
where $C$ is independent on $t$, and the fact that $|\nabla u_0^e|, |\nabla v_0^e|\in L^1$ and (\ref{reminderbound}), we obtain
\begin{eqnarray*}
|A_2|&\leq& Ct\int_{P_1^0P_2^0\backslash\mathcal{L}}|\nabla u_t^e||\nabla v_t^e|+C\int_{P_1^0P_2^0\backslash\mathcal{L}}(\nabla u_t^e\cdot\nabla v_t^e-\nabla u_0^e\cdot\nabla v_0^e)\\
&\leq& Ct\int_{P_1^0P_2^0\backslash\mathcal{L}}(|\nabla u_0^e|+t^{\beta_1})(|\nabla v_0^e|+t^{\beta_1})\\
&+&C\int_{P_1^0P_2^0\backslash\mathcal{L}}(|R_1||\nabla v_0^e(x(y),y)|+|R_2||\nabla u_0^e(x(y),y)|+|R_1||R_2|)\\
&\leq&C\left(t+\frac{t^{\gamma}}{\delta_t^{\gamma+1}}+\left(\frac{t^{\gamma}}{\delta_t^{\gamma+1}}\right)^2\right).
\end{eqnarray*}
Finally, recalling the definition of  $\delta_t$  we have
\[
|A_2|\leq Ct^{\gamma(1-\frac{\theta}{2})}.
\]
Let us now consider
\[
|A_1|=\left|\int_{P_1^tP_2^0\backslash\mathcal{L}}M_t\nabla u_t^e\cdot\nabla v_t^e(\Phi_t^{\vec V}\cdot n_t)d\sigma-\int_{P_1^0P_2^0\backslash\mathcal{L}}M_t\nabla u_t^e\cdot\nabla v_t^e(\Phi_t^{\vec V}\cdot  n_t)d\sigma\right|.
\]
Let us set 
\[
F_t=M_t\nabla u_t^e\cdot\nabla v_t^e(\Phi_t^{\vec V}\cdot  n_t).
\]
 Then, assuming without loss of generality that $P^0_1=(0,0)$ and $P^0_2=(-x_2,0)$ with $x_2>0$, we can write, for a suitable positive $c$ and $c_1$,
\[
\int_{P_1^tP_2^0\backslash\mathcal{L}}F_td\sigma=\int_{x_2+\frac{\delta_t}{c}}^{-\frac{\delta_t}{c}}F_t(x_t(\eta),y_t(\eta))\sqrt{(x'_t(\eta))^2+(y'_t(\eta))^2}d\eta
\]
and we consequently obtain
\begin{eqnarray*}
|A_1|&=&\left |\int_{x_2+\frac{\delta_t}{c}}^{-\frac{\delta_t}{c}}F_t(x_t(\eta),y_t(\eta))\sqrt{(x'_t(\eta))^2+(y'_t(\eta))^2}d\eta\right.
\\&&\left.\hphantom{aaaa}-\int_{x_2+\frac{\delta_t}{c}}^{-\frac{\delta_t}{c}}F_t(\eta,0)d\eta+\int_{I_{\delta_t}}F_t(\eta,0)d\eta\right|
\end{eqnarray*}
where $I_{\delta_t}=[-\frac{\delta_t}{c_1},-\frac{\delta_t}{c}]\cup [x_2+\frac{\delta_t}{c_1}, x_2+\frac{\delta_t}{c}]$. Note now, that by Proposition \ref{LiNirenberg} and estimates (\ref{usefulbounds}), we have
\[
|F_t(x_t(\eta),y_t(\eta))-F_t(\eta,0)|\leq \frac{C}{\delta_t^{2(\gamma+1)}}t^{\gamma}\leq Ct^{\gamma(1-\frac{\theta}{2})}
\]
and also by using Proposition \ref{BellFriedIs} we have
\[
\left|\int_{I_{\delta_t}}F_t(\eta,0)d\eta\right|\leq C\delta_t^{2\omega-1}\leq Ct^{\beta_1(2\omega-1)}.
\]
In conclusion, we derive an estimate for $J_1-J_2$ 
\begin{equation}\label{estI1I2}
|J_1-J_2|\leq Ct^{\beta}
\end{equation}
where $\beta=\min(\beta_1,\beta_1(2\omega-1), \gamma(1-\frac{\theta}{2}))$. Insertion of (\ref{estI1I2}) and (\ref{estI3I4}) into (\ref{difference}) concludes the proof.
\end{proof}

We are now ready to prove our main result.

\subsection*{Proof of Theorem \ref{MainResult}.}
We proceed in two steps. First we prove the claim for $\vec V=(\vec V_1,\vec 0,\vec 0)$ where 
$\vec V_1=(v_1,v_2)$. Assume without loss of generality that $P_1^0=(0,0)$ and $P_2^0=(-x_2,0)$ with $x_2>0$ and $P_3^0=(x_3,y_3)$ with $y_3<0$. Let $T^t$ be the triangle of vertices $P^t=P^0+t\vec V$ and let $\overline{P}^t_1$ denote the intersection of the side $P_1^tP_3^0$ with the side $P_1^0P_2^0$ axis. Assume without loss of generality $t>0$. We now decompose the displacement from $T^0$ to $T^t$ as the superposition of the dispalcement from $T^0$ to $\overline{T}^t$ of vertices $\overline{P}^t=(\overline{P}^t_1,P_2^0,P_3^0)$ and the one from $\overline{T}^t$ to $T^t$.
By Lemma \ref{sidederivative} applied in the direction of $\overline{P}^t-P^0=(\vec w^t_1,\vec 0,\vec 0)$ where $\vec w^t_1=-\lambda_t\vec e_1$ and $\lambda_t=v_2\frac{x_3-tv_1}{y_3-tv_2}-v_1$ and Remark 3.3 we get
\[
F(\overline{P}^t)-F(P^0)=(k-1)t\lambda_t\int_{\partial T^0}M_0\nabla u_0^e\cdot\nabla v_0^e(\Phi_0^{\vec W^0_1}\cdot  n_0)d\sigma+\sigma_1(t)
\]
where $\vec W_1^0=(-\vec e_1,\vec 0,\vec 0)$ and $\sigma_1(t)=o(t)$ as $t\rightarrow 0$.
Hence, also
\begin{equation}\label{firstdisplacement}
F(\overline{P}^t)-F(P^0)=(k-1)t\lambda_0\int_{\partial T^0}M_0\nabla u_0^e\cdot\nabla v_0^e(\Phi_0^{\vec W^0_1}\cdot  n_0)d\sigma+\sigma_1(t)
\end{equation}
with $\lambda_0=v_2\frac{x_3}{y_3}-v_1$. 

To compute $F(P^t)-F(\overline{P}^t)$ we now apply Lemma \ref{sidederivative} in the direction \\$\vec W_2^t=\overline{P}^t-P^t=(\vec w_2^t,\vec 0,\vec 0)$.
Then, again by Lemma \ref{sidederivative}, we can write
\begin{eqnarray*}
F(P^t)-F(\overline{P}^t)&=&F(\overline{P}^t+\mu_t\vec W_2^t)-F(\overline{P}^t)=\int_0^{\mu_t}\frac{d}{d\eta}F(\overline{P}^t+\eta\vec W_2^t)d\eta\\
&=&\int_0^{\mu_t}\left(\frac{d}{d\eta}F(\overline{P}^t+\eta\vec W_2^t)-\frac{d}{d\eta}F(\overline{P}^t+\eta\vec W_2^t)|_{\eta=0}\right)d\eta
\\&&+\mu_t\frac{d}{d\eta}F(\overline{P}^t+\eta\vec W_2^t)|_{\eta=0}
\end{eqnarray*}
and,  by Lemma \ref{continuity},
\[
\left|F(\overline{P}^t+\mu_t\vec W_2^t)-F(\overline{P}^t)-\mu_t\frac{d}{d\eta}F(\overline{P}^t+\eta\vec W_2^t)|_{\eta=0}\right|\leq Ct^{\beta}|\mu_t|\leq Ct^{\beta+1}
\]
where $\mu_t=|\overline{P}^t-P^t|=t|v_2|\sqrt{\left(\frac{x_3-tv_1}{y_3-tv_2}\right)^2+1}$ and we get
\begin{equation}\label{seconddisplacement}
F(P^t)-F(\overline{P}^t)=(k-1)t\mu_0\int_{\partial T^0}M_0\nabla u_0^e\cdot\nabla v_0^e(\Phi_0^{\vec W^0_2}\cdot n_0)d\sigma+r(t)
\end{equation}
where $|r(t)|\leq Ct^{\beta+1}$, $\mu_0=|v_2|\sqrt{\frac{x^2_3}{y^2_3}+1}$ and $\vec W^0_2=\frac{v_2}{|v_2|}\left(\frac{(\frac{x_3}{y_3},1)}{\sqrt{\frac{x_3^2}{y_3^2}+1}},\vec 0,\vec 0\right)$.

Finally, putting (\ref{firstdisplacement}) and (\ref{seconddisplacement}) together, we end up with the following formula
\begin{eqnarray*}
F(P^t)-F(P^0)&=&(k-1)t\left\{\lambda_0\int_{\partial T^0}M_0\nabla u_0^e\cdot\nabla v_0^e(\Phi_0^{\vec W^0_1}\cdot  n_0)d\sigma+\right.
\\&&\left.\hphantom{aaaa}+\mu_0\int_{\partial T^0}M_0\nabla u_0^e\cdot\nabla v_0^e(\Phi_0^{\vec W^0_2}\cdot  n_0)d\sigma\right\}+r(t)\\
&=&(k-1)t\int_{\partial T^0}M_0\nabla u_0^e\cdot\nabla v_0^e(\Phi_0^{\lambda_0\vec W_1^0+\mu_0\vec W_2^0}\cdot n_0)d\sigma+r(t)
\end{eqnarray*}
and since
\[
\lambda_0\vec W_1^0+\mu_0\vec W_2^0=\vec V
\]
the statement of step 1 follows.

Now we are left with the general case where we consider an arbitrary vector $\vec V=(\vec v_1,\vec v_2,\vec v_3)$. Let $P^t=P^0+t\vec V$; we consider now the displacement form $T^0$ to $T^t$ as superposition of three displacements. The first one for $T^0$ to $\overline{T}^t$ of vertices $\overline{P}^t=(P_1^0+t\vec v_1,P^0_2,P^0_3)$, the second one from $\overline{T}^t$ to $\tilde{T}^t$
of vertices $\tilde{P}^t=(P_1^0+t\vec v_1,P_2^0+t\vec v_2, P_3^0)$ and the third one from  $\tilde{T}^t$ to  $T^t$. We then split
\[
F(P^t)-F(P^0)=F(P^t)-F(\tilde{P}^t)+F(\tilde{P}^t)-F(\overline{P}^t)+F(\overline{P}^t)-F(P^0)=D_1+D_2+D_3
\]
The last term $D_3$ can be estimated using the result obtained in the first step of the proof. Indeed,
\[
D_3=(k-1)t\int_{\partial T^0}M_0\nabla u_0^e\cdot\nabla v_0^e(\Phi_0^{\vec V^0_1}\cdot  n_0)d\sigma+r(t)
\]
where $\vec V_1=(\vec v_1,\vec 0,\vec 0)$. For $D_2$ we proceed similarly as in step 1 decomposing the displacement from $\overline{T}^t$ to $\tilde{T}^t$ as superposition of the displacements along the sides obtaining
\[
D_2=(k-1)t\int_{\partial T^0}M_0\nabla u_0^e\cdot\nabla v_0^e(\Phi_0^{\vec V^0_2}\cdot  n_0)d\sigma+r(t)
\]
where  $\vec V_2=(\vec 0,\vec v_2,\vec 0)$ and analogously
\[
D_1=(k-1)t\int_{\partial T^0}M_0\nabla u_0^e\cdot\nabla v_0^e(\Phi_0^{\vec V^0_3}\cdot  n_0)d\sigma+r(t)
\]
where  $\vec V_3=(\vec 0,\vec 0, \vec v_3)$. Finally, summing up $D_1,D_2$ and $D_3$ we eventually get 
\[
F(P^t)-F(P^0)=(k-1)t\int_{\partial T^0}M_0\nabla u_0^e\cdot\nabla v_0^e(\Phi_0^{\vec V}\cdot  n_0)d\sigma+r(t)
\]
thus ending the proof.
\qed

\noindent
\begin{coro}\label{differenziabilita}
The map $P\to \Lambda_{\sigma_P}$ is differentiable.
\end{coro}
\begin{proof}
If we define the linear operator
\[\tilde{L}:H^{1/2}(\der\om)\to H^{-1/2}(\der\om)\]
by
\[<\tilde{L}(f),g>=(k-1)\int_{\der T^0}(M_0\nabla u_0^e\cdot\nabla v_0^e)(\Phi_0^{\vec V}\cdot n_0)\,d\sigma,\]
we can state \eqref{shapederivative} as
\begin{equation}\label{shapederivativeter}
\|\Lambda_{\sigma_{P_t}}-\Lambda_{\sigma_P}-t\tilde{L}\|_{\mathcal{L}(H^{1/2}(\der\om), H^{-1/2}(\der\om))}\leq Ct^{1+\alpha}.
\end{equation}
Since $\tilde{L}$ is linear in ${\vec V}$ and by continuity of $\tilde{L}$ with respect to $P$ (see Lemma \ref{continuity}), we actually obtained the differentiability of the map $P\to\Lambda_{\sigma_{P}}$.
\end{proof}

A similar result can be derived when considering, instead of the Dirichlet to Neumann map, the Neumann to Dirichlet map with suitable normalization conditions. In fact, consider the spaces
\[ 
H^{1/2}_{\diamond}(\partial\Omega)=\left\{f\in H^{1/2}(\partial\Omega): \int_{\partial\Omega}f=0\right\}
\]
and 
\[ 
H^{-1/2}_{\diamond}(\partial\Omega)=\{g\in H^{-1/2}(\partial\Omega): <g,1>=0\}.
\]
It is well known that the Dirichlet to Neumann map maps onto $H^{-1/2}_{\diamond}(\partial\Omega)$ and when restricted to $H^{1/2}_{\diamond}(\partial\Omega)$ it is injective and has bounded inverse. So, we can define the Neumann to Dirichlet map as the inverse of the Dirichlet to Neumann map restricted to $H^{1/2}_{\diamond}(\partial\Omega)$ i.e.
\begin{equation*}
    \mathcal{N}_{\sigma_P}: H^{-1/2}_{\diamond}\left(\der\om\right)\to H^{1/2}_{\diamond}\left(\der\om\right)
\end{equation*}
with
\[
  \mathcal{N}_{\sigma_P}=(\Lambda_{\sigma_P}|_{H^{1/2}_{\diamond}(\partial\Omega)})^{-1}.
\]
For $f,g \in H^{-1/2}_{\diamond}(\partial\Omega)$ we can define $\tilde{F}:\mathcal{V}\rightarrow \RR$ as follows
\[
\tilde{F}(P)=<g, \mathcal{N}_{\sigma_P}(f)>\quad\forall P\in \mathcal{V}.
\]
Consider now $P^0=(P_1^0,P_2^0,P_3^0) \in \mathcal{V}$ and the corresponding triangle  denoted by $T^0$. Let $\vec V=(\vec V_1,\vec V_2,\vec V_3)\in \RR^2\times\RR^2\times\RR^2$ be an arbitrary  vector and let $\Phi_0^{\vec V}: \RR^2\rightarrow\RR^2$ be the affine map that gives
\begin{equation*}
\Phi_0^{\vec V}(P_i^0)=\vec V_i,\,\,\,\, i=1,2,3.
\end{equation*}
Consider now $P^t=P^0+t\vec V=(P_1^t,P_2^t,P_3^t)$ for $t\in [0,1]$ and the triangle of vertices $P^t$ that we denote by $T^t$ and denote by $\tilde{G}(t)=\tilde{F}(P^t)$. 

Let $\sigma_0=1+(k-1)\chi_{T^0}$ and let 
$u_0\in H^1(\om)$ the unique weak solution to 
\begin{equation}\label{probPN}
    \left\{\begin{array}{rcl}
             \textrm{ div }(\sigma_{0}\nabla u_0) & = & 0\mbox{ in }\om, \\
             \frac{\partial u_0}{\partial \nu} & = & f \mbox{ on }\der\om,\\
             \int_{\der\om}u_0&=&0.
           \end{array}
    \right.
\end{equation}

Denote  by $v_0$ the solution of 
\begin{equation*}
    \left\{\begin{array}{rcl}
             \textrm{ div }(\sigma_0\nabla v_0) & = & 0\mbox{ in }\om, \\
              \frac{\partial v_0}{\partial \nu}& = & g \mbox{ on }\der\om,\\
                \int_{\der\om}v_0&=&0.
           \end{array}
    \right.
\end{equation*}
Finally, denote by $u_t:=u_{P^t}$ the solution of (\ref{probPN}) corresponding to $\sigma_t=1+(k-1)\chi_{T^t}$ and by $v_t$ the solution to 
\begin{equation*}
    \left\{\begin{array}{rcl}
             \textrm{ div }(\sigma_t\nabla v_t) & = & 0\mbox{ in }\om, \\
            \frac{\partial v_t}{\partial \nu}  & = & g \mbox{ on }\der\om,\\
            \int_{\der\om}v_t&=&0.
           \end{array}
    \right.
\end{equation*}
Let  
\[u_0^e=u_0|_{\om\backslash T^0},\,\,\,\, v_0^e=v_0|_{\om\backslash T^0}.\]

We have the following: 
\begin{teo}\label{MainResult2}  There exits two positive constants $C$ and $\alpha$ depending only on $L$, $d_0$, $\alpha_1$, $\alpha_0$, $k$, such that
\begin{equation*}
\begin{aligned}
\left|\tilde{G}(t)-\tilde{G}(0)-\vphantom{\int_{\der T^0}}\right.&\left.t(k-1)\int_{\der T^0}(M_0\nabla u_0^e\cdot\nabla v_0^e)(\Phi_0^{\vec V}\cdot n_0)\,d\sigma\right|\\
&\leq Ct^{1+\alpha}\|f\|_{H^{-1/2}(\der\om)}\|g\|_{H^{-1/2}(\der\om)},
\end{aligned}
\end{equation*}
for $|t|<d_0/(2|\vec V|)$, 
that implies
\begin{equation*}
\tilde{G}^\prime(0)=(k-1)\int_{\der T^0}(M_0\nabla u_0^e\cdot\nabla v_0^e)(\Phi_0^{\vec V}\cdot n_0)d\sigma.
\end{equation*}
\end{teo}

The proof of Theorem \ref{MainResult2} can be derived similarly as the one of Theorem \ref{MainResult} observing that all the preliminary results of Section 2 continue to hold also for solutions to the Neumann problem and that
\begin{equation*}
\frac{\tilde{G}(t)-\tilde{G}(0)}{t}=\frac{1}{t}<g,(\mathcal{N}_{t}-\mathcal{N}_{0})(f)>=\frac{1}{t}\int_{\om}(\sigma_t-\sigma_0)\nabla u_t\cdot\nabla v_0
\end{equation*}

\section{A shape derivative approach to an optimization problem}\label{section4}
We now want to see some interesting consequences of  Theorem \ref{MainResult2}.

Let $u_P\in H^1(\om)$ be the unique weak solution to 
\begin{equation*}
    \left\{\begin{array}{rcl}
             \textrm{ div }(\sigma_P\nabla u_P) & = & 0\mbox{ in }\om, \\
             \frac{\partial u_P}{\partial \nu} & = & f\mbox{ on }\der\om,\\
             \int_{\der\om}u_P&=&0
           \end{array}
    \right.
\end{equation*}
corresponding to conductivity $\sigma_P=1+(k-1)\chi_{T^P}$ and let $u|_{\partial\om}=u_{meas}$. It is well known by the results obtained in \cite{BFS} that under suitable choice of $f$ the datum 
 $u_{meas}$ uniquely determines the triangle $T^P\subset \Omega$.  

Let $T^0$ of vertices $P^0$ be an initial guess close enough to the exact solution $T^P$.
Let  $\sigma_t=1+(k-1)\chi_{T^t}$  where $T^t$ has vertices $P^t=P^0+t\vec V=(P_1^t,P_2^t,P_3^t)$ and $u_t$ is the solution to (\ref{probPN})  for  $\sigma_t=1+(k-1)\chi_{T^t}$. 
 
Consider the following boundary functional
\[
\mathcal{J}(T^t)=\frac{1}{2}\int_{\partial\om}(u_t-u_{meas})^2
\]
and consider the optimization problem of minimizing the functional with respect to all possible affine motions of the triangle $T^0$. It is well known that in order to solve this minimization problem we might use an iterative procedure like for example Newton's method which involves the computation of shape derivative of the functional $\mathcal{J}(T^t)$ (see  \cite{ABKFL}, \cite{ADKSTVZ}, \cite{DZ} and \cite{SZ}).


We now will derive a formula for the shape derivative of the functional $\mathcal{J}(T^t)$  defined as
\[
D\mathcal{J}(T^0)[\vec V]=\lim_{t\rightarrow 0}\frac{\mathcal{J}(T^t)-\mathcal{J}(T^0)}{t}.
\]
For, we will show the following preliminary result.
\begin{lm}\label{wt}
There exists positive constants $C$ and $\beta>0$ such that
\begin{equation}\label{ot}
\|u_t-u_0\|_{L^2(\partial\om)}\leq Ct^{1+\beta},
\end{equation}
for $|t|<d_0/(2|\vec V|)$.
\end{lm}
\begin{proof}
Let $w_t$ be the solution to the problem
\begin{equation*}
    \left\{\begin{array}{rcl}
             \textrm{ div }(\sigma_0\nabla w_t) & = & 0\mbox{ in }\om, \\
            \frac{\partial w_t}{\partial \nu}  & = & u_t-u_0 \mbox{ on }\der\om,\\
            \int_{\der\om}w_t&=&0.
           \end{array}
    \right.
\end{equation*}
From Theorem \ref{MainResult2} we have
\begin{eqnarray}\label{keyformula}
<g,(\mathcal{N}_{t}-\mathcal{N}_{0})(f)>&=&<g,u_t-u_0>\nonumber\\&=&(k-1)t\int_{\der T^0}(M_0\nabla u_0^e\cdot\nabla v_0^e)(\Phi_0^{\vec V}\cdot n_0)d\sigma+r(t)
\end{eqnarray}
for any $g\in H^{-1/2}(\partial\Omega)$ and where 
\begin{equation}\label{reminderbound2}
|r(t)|\leq Ct^{1+\beta}
\end{equation}
with $\beta=\min(\beta_1,\beta_1(2\omega-1), \gamma(1-\frac{\theta}{2}))$.
Choosing $g=u_t-u_0$ and since $u_t-u_0\in L^2(\partial\om)$ we can rewrite last relation in the form
\begin{equation}\label{L2norm}
\int_{\partial\om}(u_t-u_0)^2=(k-1)t\int_{\der T^0}(M_0\nabla u_0^e\cdot\nabla w_t^e)(\Phi_0^{\vec V}\cdot n_0)d\sigma+r(t).
\end{equation}
We now estimate the first term on the right hand side of (\ref{L2norm}). 
Observe that by the energy estimates 
\[
\|w_t\|_{H^1(\Omega)}\leq C\|u_t-u_0\|_{L^2(\partial\om)}\leq C\|u_t-u_0\|_{H^1(\om)}\leq Ct^{\theta}
\]
and arguing similary as in Proposition \ref{pointwisegrad} we get
\begin{equation}\label{gradwt}
\|\nabla w_t\|_{L^{\infty}(\der T^0\backslash\mathcal{L})}\leq C t^{\beta_1}
\end{equation}
We split
\begin{eqnarray*}
\int_{\der T^0}(M_0\nabla u_0^e\cdot\nabla w_t^e)(\Phi_0^{\vec V}\cdot n_0)d\sigma&=&\int_{\der T^0\backslash\mathcal{L}}(M_0\nabla u_0^e\cdot\nabla w_t^e)(\Phi_0^{\vec V}\cdot n_0)d\sigma\\
&&+\int_{\der T^0\cap\mathcal{L}}(M_0\nabla u_0^e\cdot\nabla w_t^e)(\Phi_0^{\vec V}\cdot n_0)d\sigma.
\end{eqnarray*}
By \eqref{gradwt} and by the fact that $|\nabla u_0^e|\in L^1$ we have
\[
\left|\int_{\der T^0\backslash\mathcal{L}}(M_0\nabla u_0^e\cdot\nabla w_t^e)(\Phi_0^{\vec V}\cdot n_0)d\sigma\right|\leq C\|\nabla w_t\|_{L^{\infty}(\der T^0\backslash\mathcal{L})}\int_{\der T^0\backslash\mathcal{L}}|\nabla u_0^e|\leq Ct^{\beta_1}.
\]
On the other hand proceeding as in Lemma \ref{continuity} we have that
\[
\left|\int_{\der T^0\cap\mathcal{L}}(M_0\nabla u_0^e\cdot\nabla w_t^e)(\Phi_0^{\vec V}\cdot n_0)d\sigma\right|\leq Ct^{\beta_2}=Ct^{\beta_1(2\omega-1)}.
\]

Hence, we have
\[
\left|t\int_{\der T^0}(M_0\nabla u_0^e\cdot\nabla w_t^e)(\Phi_0^{\vec V}\cdot n_0)d\sigma\right|\leq Ct^{1+\min(\beta_1,\beta_1(2\omega-1))}
\]
 and inserting last inequality and (\ref{reminderbound2}) into (\ref{L2norm}) we get (\ref{ot}).
\end{proof}
Let 
\begin{equation}\label{probw0}
    \left\{\begin{array}{rcl}
             \textrm{ div }(\sigma_0\nabla w_0) & = & 0\mbox{ in }\om, \\
            \frac{\partial w_0}{\partial \nu}  & = & u_0-u_{meas} \mbox{ on }\der\om,\\
            \int_{\der\om}w_0&=&0.
           \end{array}
    \right.
\end{equation}
We are now ready to prove
\begin{prop}\label{shapederivativefunctional}
\[
D\mathcal{J}(T^0)[\vec V]=(k-1)\int_{\der T^0}(M_0\nabla u_0^e\cdot\nabla w_0^e)(\Phi_0^{\vec V}\cdot n_0)d\sigma
\]
where $u_0$ denotes the solution to (\ref{probPN}) corresponding to $\sigma_{T^0}$ and $w_0$ the solution to (\ref{probw0}).
\end{prop}
\begin{proof}
Let us consider
\begin{eqnarray*}
\mathcal{J}(T^t)-\mathcal{J}(T^0)&=&\frac{1}{2}\int_{\partial\om}(u_t-u_{meas})^2-\frac{1}{2}\int_{\partial\om}(u_0-u_{meas})^2\\&=&\int_{\partial\om}(u_0-u_{meas})(u_t-u_0)+\frac{1}{2}\int_{\partial\om}(u_t-u_0)^2.
\end{eqnarray*}
Then,  by Lemma \ref{wt} and by (\ref{keyformula}) applied for $g=u_0-u_{meas}$ we get
\begin{eqnarray*}
\mathcal{J}(T^t)-\mathcal{J}(T^0)&=&\int_{\partial\om}(u_0-u_{meas})(u_t-u_0)+o(t)\\&=&(k-1)t\int_{\der T^0}(M_0\nabla u_0^e\cdot\nabla v_0^e)(\Phi_0^{\vec V}\cdot n_0)d\sigma+o(t).
\end{eqnarray*}
Finally, dividing last expression by $t$ and passing to the limit as $t\rightarrow 0$  the thesis follows.
\end{proof}

%
\bibliographystyle{plain}
%

\end{document}